\documentclass[11pt,a4paper]{article}

\usepackage{lmodern,amsmath,amsthm,amsfonts,amssymb,graphicx,float,microtype,thmtools,underscore,mathtools,thm-restate}
\usepackage[shortlabels]{enumitem}
\setlist[itemize]{topsep=0.5ex,itemsep=0.5ex,parsep=0.5ex}
\setlist[enumerate]{topsep=0.5ex,itemsep=0.5ex,parsep=0.5ex}
\setlist[description]{topsep=0.5ex,itemsep=0.5ex,parsep=0.5
ex}
\usepackage[usenames,dvipsnames,svgnames,table]{xcolor}
\usepackage[unicode=true]{hyperref}
\hypersetup{ 
colorlinks,
linkcolor={blue!60!black},
citecolor={black},
urlcolor={blue!60!black},
pdftitle={Rectilinear Crossing Number of Graphs Excluding Single-Crossing Graphs as Minors}}
\usepackage[capitalise, nameinlink, noabbrev]{cleveref}
\crefname{lem}{Lemma}{Lemmas}
\crefname{thm}{Theorem}{Theorems}
\crefname{cor}{Corollary}{Corollaries}

\usepackage[longnamesfirst,numbers,sort&compress]{natbib}
\makeatletter
\def\NAT@spacechar{~}
\makeatother
\usepackage[margin=2.3cm]{geometry}
\allowdisplaybreaks

\DeclarePairedDelimiter{\floor}{\lfloor}{\rfloor}

\newcommand\subsetcong{\mathrel{\text{%
    \setbox0\hbox{$\subseteq$}%
    \rlap{\hbox to \wd0{\hss\hss\hss\raisebox{1.5\height}{$\sim$}\hss}}\box0
}}}

\newcommand{\CR}[1]{\ensuremath{\protect\textsf{\textup{cr}}(#1)}}
\newcommand{\RCR}[1]{\ensuremath{\protect\overline{\textsf{\textup{cr}}}(#1)}}

\newcommand{\V}[1]{\ensuremath{\protect|#1|}} 
\newcommand{\E}[1]{\ensuremath{\protect\Vert #1\Vert}} 
\newcommand{\FFF}{\ensuremath{\mathcal{F}}}

\renewcommand{\epsilon}{\varepsilon}
\renewcommand{\emptyset}{\varnothing}

\renewcommand{\geq}{\geqslant}
\renewcommand{\leq}{\leqslant}

\newcommand{\rd}{rectilinear drawing}
\newcommand{\rds}{rectilinear drawings}
\newcommand{\rcn}{rectilinear crossing number}
\newcommand{\rcns}{rectilinear crossing numbers}
\newcommand{\scg}{single-crossing graph}

\newcommand{\gi}{\ensuremath{\protect G_i}}
\newcommand{\gim}{\ensuremath{\protect G^-_i}}
\newcommand{\gj}{\ensuremath{\protect G_j}}
\newcommand{\gjm}{\ensuremath{\protect G^-_j}}

\newcommand{\qi}{\ensuremath{\protect Q_i}}
\newcommand{\qim}{\ensuremath{\protect Q^-_i}}

\newcommand{\Oh}[1]{\ensuremath{\protect O(#1)}}

\theoremstyle{plain}
\newtheorem{thm}{Theorem}
\newtheorem{lem}[thm]{Lemma}
\newtheorem{cor}[thm]{Corollary}

\newtheorem{claim}{Claim}
\crefname{obs}{Observation}{Observations}
\newtheorem*{lem*}{Lemma}
\theoremstyle{definition}

\newtheorem*{conj*}{Conjecture}

\date{}

\begin{document}
\title{\bf\boldmath\fontsize{18pt}{18pt}\selectfont
Rectilinear Crossing Number of Graphs Excluding Single-Crossing Graphs as Minors}

\author{
Vida Dujmovi{\'c}\,\footnotemark[1]\qquad
Camille La Rose\,\footnotemark[1]\qquad
}

\maketitle

\footnotetext[1]{School of Computer Science and Electrical
  Engineering, University of Ottawa, Ottawa, Canada
  (\texttt{vida.dujmovic@uottawa.ca, claro100@uottawa.ca}). Research
  of Vida Dujmovi\'c is supported by NSERC and the University Research Chair
  grant by University of Ottawa.}

\begin{abstract}
The crossing number, \CR{G}, of a graph $G$ is the minimum number of
crossings in a drawing of $G$ in the plane. A rectilinear drawing
(also known as, straight-line drawing)  of a graph $G$ represents
vertices of $G$ by a set of points in the plane and represents each
edge of $G$ by a straight-line segment connecting its two
endpoints. The rectilinear crossing
number, \RCR{G}, of $G$ is the minimum number of crossings in a
rectilinear drawing of $G$.

The crossing lemma implies that as soon as a graph family has $n$-vertex graphs with more than $(3+\epsilon)\cdot n$ edges ($\epsilon>0)$, the best crossing number upper bound one can hope for is $O(c_\epsilon\cdot n)$. We are interested in families of graphs that admit such linear crossing number upper bounds. By the crossing lemma, the only candidates are families of graphs with a linear number of edges.

Graphs of bounded genus and bounded degree (B{\"{o}}r{\"{o}}czky, Pach
and T{\'{o}}th, 2006)
 and in fact all bounded
 degree proper minor-closed families (Wood and Telle, 2007) have been shown to
admit linear crossing number, with tight $\Theta(\Delta n)$ bound
shown by Dujmovi\'c, Kawarabayashi, Mohar and Wood, 2008. 
  Neither bounding degree nor excluding a
minor can be dropped from these statements since there are $n$-vertex
cubic graphs that have crossing number $\Theta(n^2)$; and,
$K_{3,n}$, which is a $K_5$-minor-free graph, also has crossing number $\Theta(n^2)$.

Much less is known about \rcn. It is not bounded by any function of
the crossing number. It is also known that for every $n$-vertex graph $G$ of bounded degree $\RCR{G}\in \Oh{(\CR{G}+n)\cdot \log\,n}$. Thus the above results imply that  $n$-vertex bounded degree graphs from a minor-closed family have \Oh{n\log n} \rcn.  We prove that graphs that exclude a single-crossing graph as a minor have the \rcn\ $O(\Delta n)$. 
This dependence on $n$ and $\Delta$ is best possible. A
single-crossing graph is a graph whose crossing number is at most
one. Thus the result applies to $K_5$-minor-free graphs, for
example. It also applies to
bounded treewidth graphs, since each family of bounded treewidth
graphs excludes some fixed planar graph as a minor. Prior to our work,
the only bounded degree minor-closed families known to have linear
\rcn\ were bounded degree graphs of bounded treewidth  (Wood and Telle, 2007), as
well as, bounded degree $K_{3,3}$-minor-free graphs  (Dujmovi\'c, Kawarabayashi, Mohar and Wood, 2008). In
the case of bounded treewidth graphs, our $O(\Delta n)$ result is
again tight and improves on the previous best known bound of
$\Oh{\Delta^2\cdot n}$ by  Wood and Telle, 2007 (obtained for convex geometric drawings). 

Finally, the paper demonstrates how to conduct clique-sums in a straight-line drawing setting. As such, it may provide a useful tool for eventually proving that all bounded degree minor-closed families of graphs have linear \rcn.
\end{abstract}

\section{\Large Introduction}
\label{Introduction}

In this article graphs are undirected, simple, and finite, unless
stated otherwise. For a graph $G$, with vertex set $V(G)$ and edge set
$E(G)$, let $\V{G}:=|V(G)|$ and $\E{G}:=|E(G)|$. \footnote{For each
  vertex $v$ of $G$, let $N_G(v):=\{w\in V(G):vw\in E(G)\}$ be the
  neighbourhood of $v$ in $G$. The \emph{degree} of $v$, denoted by
  $\deg_G(v)$, is $|N_G(v)|$. Let $\Delta(G)$ be the maximum degree of
  $G$. When the graph is clear from the context, we will sometimes
  write $\deg(v)$ instead of  $\deg_G(v)$ and  $\Delta$  instead of $\Delta(G)$. } 
The \emph{crossing number} of a graph $G$, denoted by \CR{G}, is the
minimum number of crossings in any drawing\footnote{A \emph{drawing}
  of a graph represents each vertex by a distinct point in the plane,
  and represents each edge by a simple closed curve between its
  endpoints, such that the only vertices an edge intersects are its
  own endpoints, and no three edges intersect at a common point
  (except at a common endpoint). A drawing is \emph{rectilinear} if
  each edge is a line-segment, and is \emph{convex} if, in addition,
  the vertices are in convex position (that is, the convex hull of the
  vertices forms a convex polygon). Rectilinear drawings are also
  known as \emph{straight-line drawings} in the literature. 
  A \emph{crossing} is a point of intersection between two edges
  (other than a common endpoint). A drawing with no crossings is
  \emph{crossing-free}. A graph is \emph{planar} if it has a
  crossing-free drawing.} of $G$ in the plane. The \emph{rectilinear
  crossing number} of a graph $G$, denoted by \RCR{G}, is the minimum number of  crossing in any \rd\ of $G$ in the plane.

Crossing number is a fundamental and extensively studied graph parameter with wide ranging applications and rich history (see the survey by \citet{zbMATH06171962} for over 700 references on the crossing number and its variants). 
	Computationally the problem, of determining the crossing number of a given graph, is notoriously difficult. Computing the crossing number is  NP-hard by Garey and Johnson \cite{gareyjohnson1983}, even for planar graph plus an edge \cite{DBLP:journals/siamcomp/CabelloM13}. It is hard to approximate even for cubic graphs \cite{DBLP:journals/dcg/Cabello13} and until recently there were no approximation algorithms with sub-polynomial in $n$ approximation factor even for bounded degree graphs \cite{DBLP:conf/stoc/ChuzhoyT22}.

On the positive side, \citet{kawarebayashireed2007} give $\Oh{f(k)\cdot \V{G}}$ algorithm for deciding whether a given graph $G$ has crossing number at most $k$.

Since computing the exact, or even an asymptotic, crossing number of a graph is hard, great deal of past research has been focused on deriving asymptotic bounds.  Regardless of the applications, be it visualization or circuit design \cite{BL84, Leighton83, Leighton84}, having as few crossings as possible is a desirable property in a drawing of a graph. This natural leads to a study on upper bounds and lower bounds on the crossing number of various graph families. 

Trivially, the crossing number of every graph $G$ is at most $\Oh{\E{G}^2}$. For some graphs this bound is asymptotically tight, including for example, the complete graph and a random cubic graph. For others, it is far from optimal. 

Every planar triangulation $G$ is known to have $3\V{G}-6$ edges (if
$\V{G}\geq 3$). Consequently, for every $n\geq 3$, there are $n$-vertex
graphs with  $3n-6$ edges that can be drawn with zero crossings. The following result, known as the \emph{crossing lemma}, tells us that as soon as a graph has a little more than $3\V{G}$ edges, it must have vastly more than zero crossings in every drawing. Specifically,  the crossing lemma, proposed by Erdos and Guy \cite{ErdosP.1973CNP},  proved by Leighton \cite{Leighton1983ComplexityII} and Ajtai et al. \cite{AjtaiM.1982CS}; and, subsequently improved \cite{Proofs3, Montaron-JGT05, PRTT-DCG06} states the following.

	\begin{thm} 
  For any $\epsilon > 0$, there exists $c_{\epsilon}$ such that, every graph $G$ with $\E{G} > (3+\epsilon)\cdot \V{G}$ edges, 
	 $$ \CR{G} \geq c_{\epsilon} \frac{\E{G}^3}{\V{G}^2} $$
	\end{thm}

An immediate consequence of this theorem is that all graphs $G$ that
have at least $(3+\epsilon)\cdot \V{G}$ edges have the crossing number
at least $\Omega (\V{G})$. We say that a family of graphs has a
\emph{linear crossing number} if there exists a constant $c$ such that
every graph $G$ in the family has $\CR{G} \leq c\cdot |G|$. The
crossing lemma tells us that for most graph families the best one can
hope of is that they have linear crossing number. By the crossing
lemma, any family of graphs that has members with the number of edges
superlinear in the number of vertices cannot have linear crossing
number.  Thus the only candidates for linear crossing number are
families of graphs whose members all have linear number of edges. One
examples of such families is a family of graphs whose members have the
maximum degree bounded by a constant. Another example is a family of
graphs that exclude some fixed graph $H$ as a minor.\footnote{Let $vw$ be
  an edge of a graph $G$. Let $G'$ be the graph obtained by
  identifying the vertices $v$ and $w$, deleting loops, and replacing
  parallel edges by a single edge. Then $G'$ is obtained from $G$ by
  \emph{contracting} $vw$. A graph $X$ is a \emph{minor} of a graph
  $G$ if $X$ can be obtained from a subgraph of $G$ by contracting
  edges.  A graph~$G$ is \emph{$X$-minor-free} if $X$ is not a minor
  of~$G$. A family of graphs \FFF\ is \emph{minor-closed} if
  $G\in\FFF$ implies that every minor of $G$ is in \FFF. \FFF\ is
  \emph{proper} if it is not the family of all graphs.} That such
families of graphs have linear number of edges follows from the
result proved independently by 
\citet{K84} and \citet{T84}. It states that $H$-minor free graphs $G$ have $O(|H|\log |H| \cdot |G|)$
edges.

It turns out that these families of graphs do not have linear crossing number. Consider, for example, the graph $K_{3,n}$. It has linear number of edges and it is $K_5$-minor-free, yet it is known to have crossing number $\Omega(n^2)$ \cite{DBLP:journals/combinatorics/Nahas03, richter1996crossing}. Similarly, consider a family of all cubic graphs. All its members have linear number of edges and yet it is known that, for every large enough $n$, there is a cubic $n$-vertex graph whose crossing number is $\Omega(n^2)$ \cite{kostochka1992bounds,diaz2003bounds,djidjev2006crossing,Leighton1983ComplexityII}.

 Thus to admit a linear crossing number, it is not enough for a family of graphs to have a bounded degree only or to only exclude a fixed graph as a minor. It turns out however that having both of these properties is enough. 
\citet{DBLP:journals/ijfcs/BorozkyPT06} first showed such a result for
bounded Euler genus graphs\footnote{Let $\Sigma$ be a surface. An
  \emph{embedding} of a graph $G$ in $\Sigma$ is a crossing-free
  drawing of $G$ in $\Sigma$. The \emph{Euler genus} of $\Sigma$
  equals $2h$ if $\Sigma$ is the sphere with $h$ handles, and equals
  $c$ if $\Sigma$ is the sphere with $c$ cross-caps. The \emph{Euler
    genus} of a graph $G$ is the minimum Euler genus of a surface in
  which there is an embedding of $G$.} as detailed in the following
theorem. Note that by the above mentioned result by \citet{K84} and
\citet{T84},  $||G|| \in
O(|G|) $ for all graphs $G$ from a proper minor closed family
of graphs. That fact will be used throughout this article, starting
with the next theorem. 

	\begin{thm}[\cite{pach2006crossing,DBLP:journals/ijfcs/BorozkyPT06}]\label{thm:pachtoth2006}
		For every integer $\gamma \geq 0$, there is a function $f$ such that every graph $G$ with Euler genus $\gamma$ has crossing number
		\[ \CR{G} \leq f(\gamma) \cdot \sum_{v \in V(G)}
                  \deg(v)^2 \leq 2\cdot f(\gamma) \cdot \sum_{v \in V(G)}
                  \in  \Oh{f(\gamma)\cdot \Delta (G) \cdot |G|} \]
	\end{thm}

	An improvement on the dependence on $\gamma$ in Theorem~\ref{thm:pachtoth2006} for orientable surfaces was shown by Djidjev and Vrt\'o \cite{djidjev2012planar}, with $\CR{G} \leq c \cdot \gamma \cdot \Delta (G) \cdot |G|$ for some constant $c$.

\citet{WT07} were the first to show that excluding a minor and bounding the maximum degree were sufficient to ensure a linear crossing number, as stated in the next theorem.

	\begin{thm}[\cite{WT07}]\label{thm:wood2006planar}
		For every graph $H$, there is a constant $c:= c(H)$ such that every $H$-minor-free graph $G$ has crossing number
		\[ \CR{G} \leq c\cdot\Delta (G)^2 \cdot |G| \]
	\end{thm}

 Theorem~\ref{thm:wood2006planar} was improved by Dujmovi\'c et al. \cite{dujmovic2018tight} by reducing the quadratic dependence on $\Delta (G)$ to linear.

	\begin{thm}[\cite{dujmovic2018tight}]\label{vidaold}
		For every graph $H$, there is a constant $c := c(H)$ such that every $H$-minor-free graph $G$ has crossing number
		\[\CR{G} \leq c \cdot \Delta(G) \cdot |G|\]
	\end{thm}

In addition, the result in Theorem~\ref{vidaold} was shown to have the best possible dependence of $\Delta(G)$ and $|G|$. These results show that we know very strong, in fact best possible, bounds on the crossing number of all proper minor-closed families of graphs of bounded degree. 

Much less is known for the \rcn. 
 
 F\'ary \cite{istvan1948straight} and Wagner \cite{wagner1936bemerkungen} proved independently that every planar graph has a \rd\ with no crossings. Hence, every planar graph $G$ has the \rcn\ 0, and thus for planar graphs $G$, $\RCR{G} = \CR{G}$.  One may be tempted to conjecture that the rectilinear crossing number and crossing number are tied. However, that is not the case. In particular, Bienstock and Dean \cite{BienstockDaniel1993Bfrc} proved that for every $m$ and every $k \geq 4$, there exists a graph $G$ with $\CR{G}=k$, but $\RCR{G} \geq m$. Therefore, Theorem \ref{thm:wood2006planar} and Theorem \ref{vidaold} do not imply that bounded degree minor-closed families of graphs have linear rectilinear crossing number. 

 In fact, in addition to planar graphs, we are only aware of the following two minor-closed families of bounded degree admitting linear \rcn. The first is the result on $K_{3,3}$-minor-free graphs by \citet{dujmovic2018tight}.
	
	\begin{thm}[\cite{dujmovic2018tight}]
		Every $K_{3,3}$-minor-free graph $G$ has rectilinear crossing number
		\[ \RCR{G} \leq \sum_{v \in V(G)} \deg(v)^2 \leq 2\cdot\Delta(G)\cdot ||G||\in \Oh{\Delta(G)\cdot |G|}\]
	\end{thm}

The second is a result on the convex crossing number\footnote{Rectilinear drawings where vertices are required to be in convex positions are called \emph{convex drawings}. For a graph $G$, the minimum number of crossing over all convex drawings of $G$ is called \emph{convex crossing number} of $G$ and is denoted by $\textup{cr}^*(G)$. Clearly, for every $G$, $\CR{G}\leq  \RCR{G} \leq \textup{cr}^*(G)$.} of bounded treewidth graphs by \citet{WT07}.

	\begin{thm}[\cite{WT07}]\label{wood2007}
		Every  graph $G$ of treewidth $k$ has convex crossing number
		\[ \textup{cr}^*(G) \in \Oh{k^2\cdot\Delta(G)^2\cdot ||G||}\in \Oh{k^3\cdot\Delta(G)^2\cdot |G|}\]
	\end{thm}

 In the case of the rectilinear crossing number a stronger bound is
 known but still with the quadratic dependence on $\Delta(G)$ in the worst case.

\begin{thm}[\cite{dujmovic2018tight}]\label{tw-rec}
  Every  graph $G$ of treewidth $k$ has \rcn
  	\[ \RCR{G} \in O(k\cdot\Delta(G)\cdot \sum_{v \in
            V(G)}\deg(v)^2) \in \Oh{k\cdot\Delta(G)^2\cdot |G|}\]
	\end{thm}

Our goal in this article is to extend this result to much wider minor-closed families of graphs of  bounded degree, in particular, a family of graphs that exclude a single-crossing graph as a minor. 
 \emph{A single-crossing graph} is a graph whose crossing number is at most one. Single-crossing minor-free graphs have been studied by the algorithms community \cite{DBLP:journals/siamcomp/EppsteinV21,DBLP:journals/algorithmica/DemaineHT05,DBLP:journals/jgaa/ChambersE13} where at times these results were precursors to algorithms and techniques applicable to more general minor-closed classes \cite{DBLP:journals/algorithmica/DemaineHT05, DBLP:journals/jacm/DemaineFHT05}.  $K_{3,3}$, $K_5$ and every planar graph are examples of single-crossing graphs. Note however that a minor of a single-crossing graph is not necessarily a single-crossing graph itself (see \cite{DBLP:journals/algorithmica/DemaineHT05} for easy examples). Note finally that a graph excluding a single-crossing graph as a minor may have arbitrarily large crossing number.  For example, any $n$-vertex graph $G$, composed of disjoint union of $\floor{\frac{n}{6}}$ copies of $K_{3,3}$, excludes $K_5$ as a minor ($K_5$ is a single-crossing graph) and yet the crossing number of $G$ is $\Theta(n)$. 

The following theorem is our main result.

	\begin{thm}\label{thm:scgrcnmainG}
		Let $X$ be a single-crossing graph. There exists a constant $c := c(X)$, such that every $X$-minor-free graph $G$ has the rectilinear crossing number of at most $c \cdot\Delta (G) \cdot |G|$.
	\end{thm}

The dependence on $\Delta$ and $|G|$ in the above theorem is best
possible. A standard lower bound constructions implies it (see for example
\cite{DBLP:journals/ijfcs/BorozkyPT06, WT07}).
 Specifically, consider a modification of the example above where this
 time the graph is comprised of the disjoint union of $\Omega(n/\Delta)$
 copies of $K_{3,3}$ where each $K_{3,3}$ is transformed into a
 $\Delta$-regular graph by adding $\Omega(\Delta)$ paths of length two
 between every pair of vertices in each copy of $K_{3,3}$. The
 resulting graph has maximum degree $\Delta$ and is still
 $K_5$-minor-free (and thus single-crossing minor-free) and yet has crossing number $\Omega(\Delta\cdot n)$. That graph also has
 treewidth at most $5$. Thus the following two corollaries are
 both tight. 

Since $K_5$ is a single-crossing graph, the following result is an immediate corollary of Theorem \ref{thm:scgrcnmainG}.

\begin{cor}\label{corK5}
There exists a constant $c$ such that every $K_5$-minor-free graph $G$ has a rectilinear crossing number of at most $c \cdot\Delta (G)\cdot |G|$.  
\end{cor}

It is known that the family of graphs of treewidth at most $k$ excludes a planar grid of size $k^c$  as a minor (for some constant $c$) \cite{ROBERTSON198692}. Since every planar graph is a single-crossing graph,  Theorem \ref{thm:scgrcnmainG} implies the following result.

\begin{cor}\label{corTW}
For every integer $k>0$, there exists a constant $c_k$ such that every graph $G$ of treewidth at most $k$ has a rectilinear crossing number at most $c_k \cdot\Delta (G)\cdot |G|$.  
\end{cor}

This corollary improves the previous best known bound on \rcn\ of
bounded treewidth graphs from $\Oh{\Delta(G)^2\cdot |G|}$ (see
Theorems \ref{wood2007} and \ref{tw-rec} above) to the optimal $\Oh{\Delta(G)\cdot |G|}$ bound. It should be noted however that Theorem \ref{wood2007} by \citet{WT07} gives  $\Oh{\Delta(G)^2\cdot |G|}$ bound for the convex crossing number of bounded treewidth graphs, and that bounds still stands as the best known for convex drawings. 

In the next section, Section~\ref{preliminaries}, we will introduce
notions that will be helpful in proving Theorem
\ref{thm:scgrcnmainG}. 
In Section~\ref{excludingscgm} we will prove Theorem \ref{thm:scgrcnmainG}. We will conclude in Section~\ref{conclusion}.

\section{Preliminaries}\label{preliminaries}

To avoid confusion on how to count crossings for the remainder of this
paper we will add to the definition of a drawing (see the footnote on
page 2) a requirement that no three edges intersect in one point unless
all three share a common endpoint. Related to that, we define a set of
points $P$ to be in \emph{general position} if no three points of $P$ lie on one line and if no three line-segments between pairs of points in $P$ intersect in one point unless all three share a common endpoint. For the ease of presentation we will add to the definition of \rds\ a condition that all endpoints of $G$ are in general position.

\subsection{Multigraphs}
Our proof of the main result, Theorem \ref{thm:scgrcnmainG}, will require the use of multigraphs. Recall that a multigraph is a graph that may have parallel edges but no loops. For the remainder of this paper, we will always employ the word multigraph when parallel edges are allowed and will use the word graph when they are not allowed, that is when the graph is simple.  The degree of a vertex $v$ in a multigraph $Q$, denoted by $\deg_Q(v)$ is the number of edges of $Q$ incident to $v$. However, unlike in simple graphs, $\deg_Q(v)$ is not necessarily equal to $|N_Q(v)|$.

A \emph{\rd} of a multigraph $Q$ represents vertices, $V(Q)$, by a set
of $|V(Q)|$ points in the plane in general position and represents
each edge by a line-segment between its endpoints. The general position
assumption implies that the only vertices an edge intersects are its
own endpoints, and no point in the drawing is in a 3 {\em distinct}
line-segments (unless all 3 share a common endpoint). It should be
noted that the parallel edges between a same pair of vertices in such
a drawing overlap, as they are represented by the same line-segment. A
\emph{crossing-pair} is a pair of edges in a \rd\ that do not have an endpoint in common and whose line-segments intersect at a common point. The number of crossings in a \rd\ of a multigraph is the number of crossing-pairs in the drawing. The \emph{\rcn} of a multigraph $Q$, denoted by $\RCR{Q}$, is the minimum number of crossings over all \rds\ of $Q$. 

Note that by these definitions, a pair of overlapping edges in a \rd\
of a multigraph is not considered a crossing-pair. That is due to the
fact that in our main proof, we eventually replace overlapping
edge-segments with edge-segments that have only one endpoint in common
and such edges can never cross. Notice also that if one is allowed to
replace line-segments by arcs in a \rd\ of a multigraph $Q$, then it
is trivial to redraw $Q$ such that the resulting "arc" drawing of $Q$
has no overlapping edges and has the same number of crossings as the
starting \rd\ of $Q$. Finally, if $Q$ is a simple graph, these definitions of \rd\ and \rcn\ are equivalent to the earlier ones for simple graphs only.

\subsection{Decompositions and Treewidth}
For graphs $G$ and $H$, an \emph{$H$-decomposition} of $G$ is a collection $(B_x \subseteq V(G) : x \in V(H))$ of sets of vertices in $G$ (called \emph{bags}) indexed by the vertices of $H$, such that
\begin{enumerate}
 \item for every edge $vw$ of $G$, some bag $B_x$ contains both $v$ and $w$, and
 \item for every vertex $v$ of $G$, the set $\{x \in V(H) 
: v \in B_x\}$ induces a non-empty connected subgraph of $H$.
\end{enumerate}

The \emph{width} of a decomposition is the size of the largest bag minus 1. The \emph{adhesion} of a decomposition is the size of the largest intersection between two bags that share an edge in $H$. If $H$ is a tree, then an $H$-decomposition is called a \emph{tree decomposition}. The \emph{treewidth} of a graph $G$ is the minimum width of any tree decomposition of $G$. Tree decomposition and treewidth are key concepts in graph minor structure theory and they have been extensively studied ever since their introduction by Halin \cite{Halin76} and independently by Robertson and Seymour \cite{RStw}.

\subsection{Rectilinear Drawings}

In the process of proving our main result, Theorem \ref{thm:scgrcnmainG} in Section \ref{sec:cls}, we will construct drawings of graphs where at one stage we will replace some vertices in those drawings with disks that fulfill certain criteria. The following lemma will be helpful for that stage.

For any positive integer $h$, let $[h]$ denote the sequence of numbers $[1, \cdots, h]$. When it is clear from the context, we will make no distinction between a vertex $v$ of a graph and the point that represents it in a drawing. Specifically, we will refer to both as $v$ when no confusion can arise. The same will be true of an edge $e$ and the line-segment representing it in a drawing.

\begin{lem}\label{lemma:diskrcn}
 Let $D$ be a \rd\ of any graph $G$. Then for each vertex $w\in V(G)$, there exists a disk $C_w$ of positive radius centered at $w$ such that the following is true. Let $v_1, \dots,v_d$ be the neighbours of $w$ in $G$. Let $P_w$ be {\em any} set of at most $d$ points in $C_w$ such that $V(G)\cup P_w$ is in general position. For each $i\in [d]$, replace the line-segment $\overline{wv_i}$ of $D$ by a line-segment between $v_i$ and any point in $P_w$. Denote that point by $p_i$. For any two $v_i$ and $v_j$ where $i\not=j$, $p_i$ and $p_j$ may not be distinct points. The resulting drawing $D'$ (of the resulting graph $G'$) has the following properties:
 \begin{enumerate}
  \item Any two edges in $G$, neither of which is incident to $w$, cross in $D'$ if and only if they cross in $D$.
  \item For each $i\in [d]$, the edge $wv_i$ and any edge $xy$ where $\{x,y\}\subseteq V(G)-\{w, v_i\}$ cross in $D$ if and only if $v_ip_i$ and $xy$ cross in $D'$.

  \item    All the remaining crossings in $D'$ are crossings between pairs of segments with distinct endpoints in $P_w$.
  \end{enumerate}
 \end{lem}

It should be noted that if $|P_w|=1$, that is if $P_w$ has exactly one point, then $D'$ is a \rd\ of $G$ where a pair of edges of $G$ cross in $D'$ if and only if they cross in $D$. 

\begin{proof}
 Start with the drawing $D$ of $G$ and a disk $C$ centered at $w$ such that the only parts of $D$ that intersect $C$ are $w$ and the edges incident to $w$. 
 Then, for each $i\in [d]$, let $S_i$ be the union of all possible line-segments from $v_i$ to any point in $C$.
 Let $S$ denote the union of all $S_i$, $i\in [d]$. 
 By reducing the radius of $C$ to some positive radius $r$ and then redefining $S$ accordingly, the following becomes true for $D$ and $C$.
 \begin{itemize}
  \item No vertex of $G$ is in $S$ other than $w,v_1,\dots,v_d$. 
  \item For each $i\in [d]$, the only vertices of $G$ that are in $S_i$ are $v_i$ and $w$.
  \item For each $i\in [d]$, the only crossing points of $D$ in $S_i$ are crossings between $wv_i$ and the edges not incident to $w$ in $G$.
  \item No segment between two crossings in $D$ is fully contained in $S$, unless it lies on one of the edges $wv_i$, $i\in [d]$.
 \end{itemize}

 Such a positive radius $r$ exists by continuity and the resulting disk meets the conditions imposed on $C_w$.
\end{proof}

\section{Main Result}\label{excludingscgm}

 In order to prove our main result, Theorem~\ref{thm:scgrcnmainG},  we will use, as one of the tools, the Robertson and Seymour's structure theorem for graphs that exclude a single-crossing graph as a minor \cite{RSonecr}. This structure theorem uses the notion of clique-sum, thus we define it next.

Let $G_1$ and $G_2$ be two disjoint graphs. Let $C_1 = \{v_1,v_2,
\cdots , v_k\}$ be a clique in $G_1$ and $C_2 = \{w_1, w_2, \cdots,
w_k\}$ be a clique in $G_2$, each of size $k$, for some integer $k
\geq 1$. Let $G$ be a graph obtained from $G_1$ and $G_2$ by
identifying $v_i$ and $w_i$ for each $i \in [k]$ and possibly deleting
some of the edges $u_i u_j$ in the resulting clique $C=\{u_1,u_2,
\cdots , u_k\}$ of $G$. Then we say that $G$ is \emph{obtained} by
\emph{$k$-clique-sums} of graphs $G_1$ and $G_2$ (at $C_1$ and
$C_2$). A \emph{$(\leq k)$-clique-sum} is an $l$-clique-sum for any
$l\leq k$. 
The following theorem by Robertson and Seymour \cite{RobertsonNeil2003GMXE} describes a structure of graphs that exclude a single-crossing graph as a minor.

\begin{thm}[\cite{RobertsonNeil2003GMXE}]\label{thm:robertson2003}
	For every \scg\ $X$, there exists a positive integer $t:=t(|X|)$ such that if $G$ is an $X$-minor-free graph, then $G$ can be obtained by $(\leq 3)$-clique-sums of graphs $G_1, \dots, G_h$ such that for each $i\in [h]$, $G_i$ is a planar graph (with no separating triangles) or the treewidth of $G_i$ is at most $t$.
\end{thm}

The graphs $G_1,\cdots, G_h$ in Theorem \ref{thm:robertson2003} are called the \emph{pieces} of the decomposition. 

Theorem\footnote{Note that the original statement of Theorem \ref{thm:robertson2003}  by Robertson and Seymour \cite{RobertsonNeil2003GMXE} does not mention separating triangles. The reason such a statement can be made is that any planar graph $G$ containing a separating triangle can itself be obtained by $3$-clique-sums of two strictly smaller planar graphs, $G_1$ and $G_2$, where the clique-sum is performed on that separating triangle.}
 \ref{thm:robertson2003}  is equivalent to stating that every
 $X$-minor-free graph $G$ has a tree decomposition of adhesion at most
 $3$ such that the vertices in each bag of the decomposition induce in
 $G$ either a planar graph (with no separating triangles) or a graph
 of treewidth at most $t$. Armed with these notions, we are now ready
 to state a more precise version of our main result  (Theorem~\ref{thm:scgrcnmainG}).

\begin{thm}\label{thm:scgrcnmain}
Let $X$ be a single-crossing graph. Let $G$ be an $X$-minor-free graph
and let $t:=t(|X|)$ be the integer from Theorem \ref{thm:robertson2003}. 
Then $G$ has a rectilinear crossing number at most $3\cdot(t^2+ 2t+2)\cdot\Delta(G)\cdot||G||$.
\end{thm}

 Theorem~\ref{thm:scgrcnmain} is a strengthened version of
 Theorem~\ref{thm:scgrcnmainG} by Theorem  \ref{thm:robertson2003} and
 the fact that $\E{G}\in \Oh{|X|\sqrt{\log{|X|}}\cdot \V{G}}$
 \cite{K84,T84} (as discussed earlier). Hence, the remainder of this
 section will be dedicated to proving Theorem~\ref{thm:scgrcnmain}. To
 do so, one has to be able to produce \rds\ of the pieces, $G_1,
 \dots, G_h$, of the decomposition (from Theorem
 \ref{thm:robertson2003}) with the claimed number of crossings and
 then combine these drawings by conducting clique-sums. The following is a sketch of the two main steps our proof will take.

\begin{enumerate}
 \item[Step 1.] Foremost, Theorem~\ref{thm:scgrcnmain} has to be true for the pieces $G_i$ of the decomposition, namely the planar graphs and bounded treewidth graphs. By the F\'ary-Wagner theorem \cite{istvan1948straight,wagner1936bemerkungen}, we know that Theorem~\ref{thm:scgrcnmain} is true for all planar graphs. In fact, it is true with bound zero for the rectilinear crossing number. On the contrary, if $G_i$ is a bounded treewidth graph, the required $O(\Delta(G_i)\cdot |G_i|)$ bound on its rectilinear crossing was not known prior to our work. Thus one of the goals of this paper is to prove that bound for bounded treewidth graphs as one of the necessary steps in the proof of Theorem~\ref{thm:scgrcnmain}. 

 \item[Step 2.] Suppose now that for each piece, $G_i$ of the
   decomposition,  we have already established the $O(\Delta(G_i)\cdot
   |G_i|)$ bound for the rectilinear crossing of $G_i$. The main goal
   then becomes demonstrating that the \rds\ of $G_1, G_2,\dots, G_h$ can be joined by performing  clique-sums without increasing the number of crossings in the final drawing of $G$ by too much. Prior to this work it was not known how to conduct clique-sums on \rds\ while achieving that goal. In particular, we need to join \rds\ of  $G_1, G_2,\dots, G_h$ in such a way that the resulting number of crossings in the \rd\ of $G$ is $O(\Delta(G)\cdot |G|)$.
\end{enumerate}

The main challenge for proving Theorem~\ref{thm:scgrcnmain} is Step 2 above. To overcome that challenge, we introduce the notion of simplicial blowups of graphs. The use of these simplicial blowups however impacts Step 1. In particular, it is not enough any more to prove that the pieces of the decomposition have $O(\Delta(G_i)\cdot |G_i|)$ \rcn. We must prove a stronger condition, namely that the simplicial blowups of the pieces have such a \rcn. 

In Section~\ref{sec:cls} we introduce simplicial blowups and
demonstrate how to achieve Step 2. In Section \ref{sec:tooldraw} we introduce graph partitions and present a helpful lemma for producing \rds. In Section~\ref{sec:planar} and \ref{sec:tw}, we then prove that Step 1 above can be accomplished, or more precisely that simplicial blowups of planar graphs and bounded treewidth graphs have the desired \rcn. Once those two steps have been achieved, we will conclude the proof of Theorem~\ref{thm:scgrcnmain} in Section~\ref{section:mainproof}.

\subsection{Bound for Rectilinear Crossing Number Using Clique-Sums}\label{sec:cls}

	A multigraph $Q$ is called a \emph{$(\leq k)$-simplicial blowup} of a graph $G$ if $Q$ can be obtained from $G$ by adding an independent set of vertices $S$ to $G$, and performing the following steps for each vertex $u$ in $S$:
\begin{enumerate}
	\item Make $u$ adjacent to all the vertices of some clique of
          size at most $k$ of $G$ 
    \item Add  zero or more parallel edges between $u$ and its neighbours in $G$.
\end{enumerate}
and finally, once Steps 1 and 2 are conducted on all vertices of $S$, delete zero or more edges from each clique of $G$ involved in Step 1.

Theorem~\ref{thm:clsrcn} is the key technical tool of this paper. It
shows how \rds\ (of simplicial blowups) of the pieces of a
decomposition can be combined into a \rd\ of a graph obtained by
clique-sums of the pieces, all while not increasing the final number
of crossings by too much. The previous result  on the crossing number
of minor-closed families (Theorem \ref{vidaold}) by \citet{dujmovic2018tight}, also had to deal with performing clique-sums on drawings while controlling the crossing number. Our proof of Theorem~\ref{thm:clsrcn} is inspired by their proof. However the drawings produced by their theorem have many bends per edge are thus are far from rectilinear drawings.

The following theorem is stated in a form that is more general than we will require. Specifically, the theorem does not require the pieces \gi\ of the decomposition to be planar or of bounded treewidth. As such, Theorem \ref{thm:clsrcn} may be useful in future work on \rcn\ of $X$-minor-free graphs where $X$ is not necessarily a single-crossing graph and thus the pieces of the decomposition are the almost embeddable graphs from the Robertson and Seymour graph minors theory.

We say that a graph $R$ is \emph{$(k,c)$-agreeable} if for every induced subgraph $R'$ of $R$ and every $(\leq k)$-simplicial blowup $R^*$ of $R'$, $\RCR{R^*}\leq c \cdot \Delta (R^*)\cdot ||R^*||$.

\begin{thm}\label{thm:clsrcn}
Let $c$ be a positive number, $k$ a positive integer, and $G_1, \cdots, G_h$ a collection of graphs such that every \gi\ is $(k,c)$-agreeable. Then every graph $G$ that can be obtained by $( \leq k )$-clique-sums of graphs $G_1, \cdots, G_h$ has rectilinear crossing number $\RCR{G}\leq k\cdot (c+2)\cdot \Delta (G)\cdot ||G||$.
\end{thm}
\begin{proof}
  
Since clique-sums identify vertices, to avoid confusion, we will
assume that the vertices of the final graph $G$ have names and that
each vertex in each piece $G_i$, $i\in[h]$  inherits its name from
$G$. Thus vertices that are identified by clique-sums have the same
name in the pieces involved. Consequently, there may be multiple vertices with the
same name in the disjoint union of $G_1, G_2, \dots G_h$. 
  
We may assume that the indices $1,\dots,h$ are such that for all $j
\geq 2$, there exists a minimum $i$ such that $i < j$ where $G_i$ and
$G_j$ are joined at some clique $C$ of $G_i$ when constructing $G$. We
define $G_i$ to be the \emph{parent} of $G_j$, with $P_j = V(C)$ being
the \emph{parent clique} of $G_j$. The parent clique of $G_1$ is the
empty set. Note that, by the introductory paragraph of this proof, it makes sense
to talk about the clique $P_j$ as existing in both in $G_i$ and in $G_j$.

Let $T$ be a rooted tree with vertex set $\{1,\dots,h\}$, where $ij$ is an edge of $T$ if and only if $G_j$ is a child of $G_i$. Let $T_i$ denote the subtree of $T$ rooted at $i$ and $U_i$ be the set of the children of $i$ in $T$. 

For each $i \in [h]$, let $G^{-}_{i} = G_i - P_i$. Note that for each $v \in V(G)$, there is exactly one $i \in [h]$ such that $v$ is in $V(G^{-}_{i})$.  Thus $V(G^{-}_{1}), \cdots , V(G^{-}_{h})$ is a partition of $V(G)$. We say that a vertex $v$ of $G$ \emph{belongs} to vertex $i$ of $T$ if $v\in\gim$.
For each $i\in [h]$, let $G[T_i]$ denote the graph induced in $G$ by the vertices of $G$ that belong to the vertices of $T_i$, that is the graph induced in $G$ by $\bigcup\{V(\gjm)\, :\, j\in T_i\}$.

\noindent\textbf{Defining the $\mathbf{(\leq k )}$-simplicial blowups
  of pieces.} To prove the theorem, we now define, for each \gim, $i \in [h]$, a specific $(\leq k
)$-simplicial blowup, denoted by \qim. To define \qim, start with
\gim. For each child $G_j$ of $G_i$, add a new vertex $c_j$ to
\gim. We call $c_j$ a \emph{dummy} vertex and say that $c_j$
\emph{represents} $G_j$ in \qim. Note that for all $j\in [2, \cdots,
h]$, there is exactly one $i<j$ such that \qim\ has a vertex that
represents \gj\ (namely, the vertex $c_j$). For the clarity of the
next statement, note first that
$V(\gim)\cap P_j$ is not empty as otherwise $P_j$ would also exist in some
$G_f$ where $f<i$ and $G_j$ would not be a child of $G_i$.  For each edge $vw \in
E(G)$, where $v \in V(\gim)\cap P_j$ and $w \in G^{-}_{\ell}$, where
$i<\ell$ and $\ell \in V(T_j)$, connect $v$ to $c_j$ by an edge. Label
that edge with the triple $(v, w, \mathcal{P}_{vw})$, where
$\mathcal{P}_{vw}$ is the path in $T$ from $i$ to $\ell$. We call the
edge labelled $(v, w, \mathcal{P}_{vw})$ in \qim\ an \emph{isthmus}
edge. It represents the edge $vw$ in the final drawing of $G$. We consequently refer to the edge $vw$ of $G$  as isthmus edge as well.
 We say that two isthmus edges are \emph{siblings} if they are adjacent
to the same dummy vertex. For a vertex $u$ in \qim\ such that  $u$ is in $P_j$ for some child
$G_j$ of $G_i$,  we say that $u$ is \emph{involved} in a clique-sum in
\qim. 
Thus each isthmus edge of \qim\ has an endpoint in \gim\ that is involved in some clique-sum in \qim.
We finally remove from \qim\ the edges in $P_j$ that are not in $G$. 
 We set the resulting multigraph to be \qim. 
 
Notice that \qim\ is a $(\leq k)$-simplicial blowup of \gim.  Since
\gi\ is $(k,c)$-agreeable (by the assumption) and since \gim\ is an induced subgraph of \gi, it follows that $\RCR{\qim}\leq c \cdot \Delta (\qim)\cdot ||\qim||$.

For $i\in [h]$, consider a \rd\ of \qim\ with at most $c \cdot \Delta
(\qi)\cdot ||\qi||$ crossings. We will construct the desired \rd\ of
$G$ by joining these \rds\ of \qim. Consider for a moment solely the
disjoint union of these \rds. The resulting \rd\ of the disjoint union
has most $\sum_{i\in [h]} c \cdot \Delta (\qim)\cdot ||\qim||$
crossings.

Notice that there is one-to-one mapping between the edges of $G$ and
the edges in the union of all $Q^-_1, Q^-_2, \dots, Q^-_h$, that is in
$\bigcup_{i\in [h]} E(\qim)$ (where the isthmus edges of $G$ map to
the isthmus edges in the union and where the non-isthmus edges of $G$
map to the  non-isthmus edges of the union). Thus $||G||=\sum_{i\in
  [h]} ||\qim||$. Hence, if for all $i\in [h]$, $\Delta (\qim)\leq
k\cdot\Delta(G)$, the above sum would be upper bounded by $c\cdot k
\cdot \Delta(G)\cdot \sum_{i\in [h]} ||\qim||=c\cdot
k\cdot\Delta(G)\cdot ||G||$. This is akin to the upper bound that we
want on the rectilinear crossing number of $G$.  
Thus we want to first bound the degree of each vertex in \qim\ by
$k\cdot\Delta(G)$. This is not completely obvious due to the addition
of the dummy vertices in the construction of \qim\ and also due to the
fact that  clique-sums allow for edge deletions from the cliques. 

	\begin{claim}\label{claim:sbdegree}
		For every $i\in [h]$ and every $v \in\qim$, $\deg_{\qim} (v) \leq k \cdot \Delta (G)$. 
	\end{claim}
	\begin{proof}
There are three cases to consider.
		\begin{description}
			\item[Case 1.] {\em $v$ is a dummy vertex of \qim}. \\
		By construction, for some $j\in U_i$, $v$ represents some \gj\ and is adjacent to at most $k$ vertices of the parent clique $P_j$ in \gi. Each edge between a vertex $u \in P_j$ and $v$ corresponds to an (isthmus) edge in $G$ adjacent to $u$. Since $\deg_G(u)\leq\Delta(G)$, $v$ is incident to at most $k \cdot \Delta (G)$ edges, giving $\deg_{\qim} (v) \leq k \cdot \Delta (G)$.
			\item[Case 2.] {\em $v$ is in \gim\ (and thus not a dummy vertex) and $v$ is not involved in any clique-sums}. Then, it follows that $\deg_{\qim} (v) \leq \deg_G (v)$.
			\item[Case 3.] {\em $v$ is in \gim\ (and thus not a dummy vertex) and is involved in at least one clique-sum}.\\
				
 Consider every $j \in U_i$ such that $v \in P_j$. Then $v$ has a
 least one neighbour in $G[T_j]$ and thus at least one edge connecting
 it to $c_j$, otherwise the clique-sum could have omitted $v$.   
Additionally, there exists a one-to-one mapping between the set of
edges in $G$ between $v$ and its neighbours in $G[T_j]$ and the set of
(parallel isthmus) edges between $v$ and $c_j$ in \qim. 
In other words, there is a one-to-one mapping between the isthmus edges incident to $v$ in \qim\ and the isthmus edges incident to $v$ in $G$. Finally, consider the non-isthmus edges incident to $v$ in \gim.
Each edge of $G_i$ that has been removed in the construction of \qim\
(namely the edges removed from $P_j$) was also removed in $G$, thus
$\deg_{\qim} (v) \leq \deg_G (v)$.
\end{description}
\end{proof}

With degrees of the vertices of \qim sorted out, we are ready to
describe how to construct a \rd\ of $G$ from the \rds\ of $Q^-_1, Q^-_2, \dots, Q^-_h$.

\noindent\textbf{Constructing the \rd\ of $\mathbf{G}$ from the \rds\ of $\mathbf{Q^-_1, Q^-_2, \dots, Q^-_h
}$.}
Since  for each $i\in [h]$, \qim\ is $(k,c)$-agreeable, $\RCR{\qim}\leq c \cdot \Delta (\qim)\cdot ||\qim||$. By Claim \ref{claim:sbdegree}, $\RCR{\qim}\leq c\cdot k \cdot \Delta (G)\cdot ||\qim||$. Let $D(\qim)$ denote a \rd\ of \qim\ with at most $c \cdot k \cdot \Delta (G) \cdot ||\qim||$ crossings.  For the remainder of the proof, we will show how to construct a rectilinear drawing, $D(G)$, of $G$ by combining the \rds\ $D(\qim)$ of \qim, $i \in [h]$, such that the resulting number of crossings in $D(G)$ is as claimed in the theorem. 

Note that removing dummy vertices (and their incident isthmus edges)
from $D(\qim)$ gives a \rd\ of \gim. Denote these \rds\ by
$D(\gim)$. In the final drawing, $D(G)$, the drawing of each \gim\
will be identical to $D(\gim)$, possibly scaled and/or rotated. In
other words, in $D(G)$, the implied \rd\ of the disjoint union of $G^-_1, G^-_2, \dots, G^-_h$ will be the disjoint union of $D(Q^-_1), D(Q^-_2), \dots, D(Q^-_h)$. Only the isthmus edges will be redrawn in this construction.

 We will join the \rds\ $D(\qim)$, $i \in [h]$ in the order of their
 indices. For $\ell\in [h]$, $D_\ell$ denotes the \rd\ obtained by
 joining $D(Q^-_1), D(Q^-_2), \dots, D(Q^-_\ell)$ (joining is detailed
 below). The \rd\ $D_h$ will thus be the desired \rd\ $D(G)$ of
 $G$. While joining these drawings, we will maintain the invariant
 that for each $j>\ell$, such that the parent of $G_j$ is some $G_i$
 with $i\in[\ell]$, the \rd\ $D_\ell$ contains the representative
 dummy vertex ($c_j$) of each \gj. Furthermore we maintain that
 $D_\ell$ minus the dummy vertices (that is $D_\ell-\cup_{j>\ell}
 c_j$) is isomorphic to $G-\cup_{j>\ell} V( G[T_j])$.

 We start by defining $D_1=D(Q^-_1)$. $D_1$ satisfies the above invariant. For $j\in [2, \cdots, h]$ we construct $D_\ell$ from $D_{\ell-1}$ and $D(Q^-_\ell)$ as follows. By the invariant, $D_{\ell-1}$ has a dummy vertex $c_\ell$ representing $G_\ell$. Let $C_\ell$ be a disk centered in the point $c_\ell$ in $D_{\ell-1}$ that meets the conditions of Lemma \ref{lemma:diskrcn}. Let $v_1,v_2,\dots,v_d$ be the neighbours of $c_\ell$ in $D_{\ell-1}$.
 Construct $D_{\ell}$ by  following steps.

 \begin{enumerate}
 \item Remove $c_\ell$ and its incident (isthmus) edges. 
 \item Scale down $D(Q^-_\ell)$. Place it inside $C_\ell$ and rotate
   it such that all the vertices of $D_\ell$ are in general position.  
 \item For each isthmus edge labelled with $(x, y, P_{xy})$ that was incident to $c_\ell$, (re)draw it as the line-segment from $x$ to $y$ if $y$ in $Q^-_{\ell}$. Otherwise, by construction, $D(Q^-_\ell)$ has a point $c_j$, $j>\ell$ and $y\in G[T_j]$. In that case, draw a line-segment between $x$ and $c_j$. By Lemma \ref{lemma:diskrcn}, the only new crossings (pairs) that this introduces are crossings between (a) a pair of (re)drawn sibling isthmus edges (that were both incident to $c_\ell$) or (b) one such isthmus edge (incident to $c_\ell$) and edges strictly inside the disk $C_\ell$ (that is, edges in $G_\ell^-$).
 \end{enumerate}

 The resulting drawing $D_\ell$ satisfies the invariant. Note that at the end of this process, when $\ell=h$, there are no more dummy vertices and each edge labelled $(x, y, P_{xy})$ in $D_h$ is an actual line-segment connecting vertex $x$ and $y$ in $G$ and thus actually represents the isthmus edge $xy$ of $G$.
 The final drawing $D_h$ is a \rd\ $D(G)$ of $G$. It remains to prove that $D(G)$ has the claimed number of crossings.
 
Before joining the drawings $D(\qim)$, $i \in [h]$, the total number of crossings in the disjoint union of all drawings was at most $c\cdot k \cdot \Delta(G)\cdot ||G||$, as argued earlier. We name this quantity the \emph{initial sum}. We now prove that joining these drawings into a drawing of $G$ does not increase the initial sum by much. Specifically, we will show that all new crossings can be charged to the edges of $G$ such that each edge is charged at most $2\cdot k\cdot \Delta(G)$ new crossings, which will complete the proof.
 
 By the construction, the new crossings must involve at least one isthmus edge. Consider such an isthmus edge $e$ labelled $(v, w, P_{vw})$, where $v\in \qim$ and $w\in Q^-_p$, $i<p$ and $w\in G[T_j]$ where $j\in U_i$ (and thus $p\in T_j$). There are four cases to consider.

\begin{description}
    \item[Case 1.] Consider first a crossing in $D(G)$ between $e$ and
      a non-isthmus edge $e'$ in \qim. That crossing is already
      accounted for in the initial sum by the crossing in $D(\qim)$
      between $e'$ and edge $vc_j$ labelled $(v,w, P_{vw})$.

    \item[Case 2.] Consider next a crossing between $e$ and an isthmus
      edge $e'$ labelled $(x, y, P_{xy})$, where $x\in \qim$ and $y\in
      G[T_r]$, with $r\in U_i$. 2a) If $r\not=j$ (so $e$ and $e'$ are
      not sibling isthmus edges),  then crossing between  $e$ and $e'$ was accounted for as
      well in the initial sum by the crossing in $D(\qim)$ between
      the edge $vc_j$ labelled  $(v,w, P_{vw})$ and the edge $xc_r$
      labelled $(x, y, P_{xy})$. 2b) If $r=j$,     it must be that $v\not=x$ as otherwise $e$ and $e'$ cannot
    cross. By construction both $w$ and $y$ are in the disk $C_r$. In
    the construction of $G$, $G[T_j]$ is added via a $(\leq
    k)$-clique-sum to $G_i$ (with parent clique $P_j$). Thus at most
    $k\cdot\Delta(G)$ (isthmus) edges cross the cycle bounding
    $C_r$. Thus $e$ can be crossed by at most $k\Delta(G)$ such edges
    $e'$. We charge these at most $k\cdot\Delta(G)$ crossings to $e$.
  
    \item[Case 3.] Consider next a crossing between $e$ and any
      edge $e'$ where both endpoints of $e'$ are in $G[T_j]$. The
      endpoints of $e'$ are thus in $Q^-_{a}$ and $Q^-_{b}$ where
      $j\leq a\leq b$. We charge the crossing to $e'$. (Think of that
      crossing being charged to $e'$ in $Q^-_{a}$). As argued above,
      at most $k\cdot\Delta(G)$ (isthmus) edges cross the cycle $C_a$
      that replaced the dummy vertex $c_a$ thus each such edge $e'$ is
      charged at most $k\cdot\Delta(G)$ new crossings.

      \item[Case 4.] Finally consider a crossing between $e$ and an isthmus
      edge $e'$ labelled $(x, y, P_{xy})$, where $x\in Q^-_f$ with $f
      <i$. Then there exists $g\in U_f$ such that $i\in T_g$. In that
      case both endpoints of $e$ are in $G[T_g]$ and we are in Case 3
      with the roles of $e$ and $e'$ reversed. Thus at most
      $k\cdot\Delta(G)$ crossings are charged to $e$.

\end{description}

 By the arguments above, each edge of $G$ is charged at most $2\cdot
 k\cdot\Delta(G)$ new crossings (at most $k\cdot\Delta(G)$ in Case 2b
 and at most $k\cdot\Delta(G)$ in Case 4)
 . Thus together with the initial sum that gives the total number of
 crossings of at most  $(c+2)\cdot k\cdot \Delta(G) \cdot ||G|| $, as claimed.
\end{proof}

\subsection{Rectilinear Drawings of Multigraphs via Graph Partitions}\label{sec:tooldraw}

As mentioned previously, in order to prove our main result, (Theorem
\ref{thm:scgrcnmain}), we will use as a main tool the theorem that we
have just proved, Theorem \ref{thm:clsrcn}. Theorem
\ref{thm:robertson2003} tells us that in order to use Theorem
\ref{thm:clsrcn}, we need to show that planar graphs and bounded
treewidth graphs are $(3,c)$-agreeable for some constant $c$. In this section, we define graph partitions and prove a lemma that will be helpful in proving that planar graphs are $(3,c)$-agreeable in Section~\ref{sec:planar} and that bounded treewidth graphs are $(k,c)$-agreeable for any $k\geq 1$ in Section~\ref{sec:tw}.

An $H$-\emph{partition} of a (multi)graph $G$ is comprised of a graph $H$ and a partition of vertices of $G$ such that
\begin{itemize}
\item each vertex of $H$ is a non-empty set of vertices of $G$ (called a \emph{bag}),
\item every vertex of $G$ is in exactly one bag of $H$, and
\item if an edge of $G$ has one endpoint in $A$ and the other endpoint in $B$ and $A$ and $B$ are distinct, then $AB$ is an edge of $H$.
\end{itemize}

The \emph{width} of a partition is the maximum number of vertices in a bag. The \emph{density} of a bag of an $H$-partition is the number of edges of $G$ with at least one endpoint in that bag. The \emph{density} of an $H$-partition is the maximum density over all bags of $H$. A bag is said to be \emph{solitary} if it contains exactly one vertex of $G$.

The proof of the following lemma is a slight modification of a similar result by \citet{WT07} and can be found in the appendix. 

\begin{lem}\label{lemma:hpartitionrcn}
 Let $K$ be a multigraph and $H$ a simple graph such that $K$ has an $H$-partition of width $w$ and density $d$. Let $X$ be the set of all vertices of $K$ that are not in solitary bags of $H$. Then we have the following.
	
	\begin{enumerate}
		\item $\RCR{K} \leq \RCR{H} \cdot w^2\cdot\Delta(K)^2 + (w-1) \cdot \sum_{v\in X} \deg_K(v)^2$
		\item if $H$ is planar, then \\
 (a) there exists a rectilinear drawing of $K$ with a most $2 \cdot d$ crossings per edge.\\
 (b) if in addition, the non-solitary bags of $H$ form an independent set in $H$, then there is a \rd\ of $K$ with at most $d$ crossings per edge.
	\end{enumerate} 
\end{lem}

\subsection{Rectilinear Crossing Number of Simplicial Blowups of Planar Graphs}\label{sec:planar}

Theorem \ref{thm:robertson2003} tells us that in order to use Theorem \ref{thm:clsrcn}, it is enough to consider $(\leq 3)$-simplicial blowups of planar graphs with no separating triangles.
In other words, it is enough to prove that planar graphs with no  no separating triangles are $(3,c)$-agreeable for some constant $c$. The next lemma achieves that.

\begin{lem}\label{lemma:separatingtrianglesrcn}
Every planar graph $G$ that has no separating triangles is $(3,3)$-agreeable. 
\end{lem}

\begin{proof}
Since every induced subgraph of $G$ is also planar and with no separating triangles, it is enough to show that every $(\leq 3)$-simplicial blowup $Q$ of $G$ has rectilinear crossing number $\RCR{Q}\leq 3 \cdot \Delta (Q)\cdot ||Q||$.

Let $S = V(Q)- V(G)$. Since adding a $1$-simplicial or $2$-simplicial
vertex to a planar graph results in a planar graph, we may assume that
each vertex in $S$ has exactly 3 neighbours in $G$. We now define an
$H$-partition of $Q$. To start, we make $H$ isomorphic to $G$ and put
each $v \in V(G)$ in the bag $B_v$ in $H$. Currently, all the bags in
$H$ are solitary bags. Since $G$, and therefore the current $H$, has
no separating triangles and since $S$ is an independent set in $Q$, we
have that for each $v \in S$, $N_{Q}(v)$ induces a face in an
embedding of $G$ and thus it is a face in the equivalent embedding of
$H$. For each vertex set $\{x,y,z\}$ in $H$ that forms such a face, we
add a bag $B_{xyz}$ adjacent to $x$, $y$ and $z$ in $H$. The resulting
graph $H$ is simple and planar. For each vertex $v \in S$ adjacent to
$x$, $y$ and $z$ in $Q$, add $v$ to the corresponding bag $B_{xyz}$ in
$H$. Thus the defined graph $H$ and the assignment of the vertices of
$Q$ to its bags defines an $H$-partition of $Q$.  

As every vertex of $Q$ in bag $B_{xyz}$ is adjacent to all vertices in
$\{x,y,z\}$, the maximum number of edges of $Q$ with an endpoint in a
non-solitary bag $B_{xyz}$ is at most
$\deg_Q(x)+\deg_Q(y)+\deg_Q(z)\leq 3\cdot\Delta (Q)$. The maximum number of
edges of $Q$ with an endpoint in a solitary bag of $H$ is clearly
$\Delta(Q)$. Thus the density of the $H$-partition is at most
$3\cdot\Delta (Q)$. Additionally, the non-solitary bags of $H$ form an
independent set in $H$ which,  by Lemma~\ref{lemma:hpartitionrcn}  (2b), implies that $Q$ has a rectilinear drawing with at most $3\cdot\Delta (Q)$ crossings per edge, giving the desired result, $\RCR{Q} \leq 3 \cdot \Delta (Q) \cdot ||Q||$. 
\end{proof}

\subsection{Rectilinear Crossing Number of Simplicial Blowups of Treewidth-\emph{k} Graphs}\label{sec:tw}

In this section, we prove that bounded treewidth graphs are
$(k,c)$-agreeable for some constants $k$ and $c$. We start with the following trivial bound applicable to all graphs.

\begin{lem}\label{lemma:multigraphsbrcn}
Every graph $G$ is $(|G|,\ |G| -1)$-agreeable. 
\end{lem}
\begin{proof}

If $|G|=1$, the statement is trivial since every $(\leq 1)$-simplicial blowup of $G$ is a star thus the crossing number of every such blowup is zero. Assume now that $|G|\geq 2$. Since every induced subgraph of $G$ is also in the class of all graphs, it is enough to show that every $(\leq|G|)$-simplicial blowup $Q$ of $G$ has rectilinear crossing number $\RCR{Q}\leq (|G|-1) \cdot \Delta (Q)\cdot ||Q||$.

Let $S = V(Q) - V(G)$. We build an $H$-partition of $Q$ as
follows. Start with $H:=K_2$ with $V(H)=\{v,w\}$. Place one vertex of
$G$ in $B_v$ and all the remaining vertices of $G$ in $B_w$. Add a
independent set of $|S|$ of vertices to $H$ and make each connected to
$v$ and $w$. It is simple to verify that $H$ is a simple planar
graph.  
 Place each vertex of $S$ in a new vertex (bag) of $H$. That defines
 an $H$-partition of $Q$ where $H$ is a simple planar graph and where
 all bags of $H$ are solitary except for one bag, that is $B_w$. 
Trivially, that one non-solitary bag forms an independent set in
$H$. Since $H$ is planar and since the density of $H$ is at most
$(|G|-1)\cdot\Delta(Q)$, we obtain the desired result, $\RCR{Q} \leq
(|G| - 1) \cdot \Delta(Q)\cdot ||Q||$  by Lemma~\ref{lemma:hpartitionrcn} (2b).
\end{proof}

The following result, obtained by setting $|G|=t$, is an immediate corollary of Lemma \ref{lemma:multigraphsbrcn}.

\begin{cor}\label{corollary:completesbrcn}
The complete graph, $K_t$, is $(t,\ t-1)$-agreeable.
\end{cor}

We are now ready to prove that every bounded treewidth graph $G$ has $\RCR{G} \in\Oh{\Delta(G)\cdot |G|}$.

	\begin{thm}\label{thm:twrcn}
For $k\geq 1$, let $\mathcal G$ denote a family of graphs of treewidth at most $k$. For every graph $G\in {\mathcal G}\,$, $\RCR{G} \leq k\cdot (k+2) \cdot \Delta (G) \cdot ||G||$.
	\end{thm}
	\begin{proof}
 It is well known (see \cite{BODLAENDER19981} for example) that $G$ can be obtained by $(\leq k)$-clique-sums
 on graphs $G_1, G_2 \cdots$, where each $G_i$, is the complete graph
 on at most $k+1$ vertices. 
Corollary \ref{corollary:completesbrcn} implies that, for each $i$, $i\in[h]$, \gi\ is $(k+1,\ k)$-agreeable and thus $(k,\ k)$-agreeable. This fulfills the sole condition of Theorem~\ref{thm:clsrcn}. Thus $\RCR{G} \leq k\cdot (k+2) \cdot \Delta (G) \cdot ||G||$.
\end{proof}

        Theorem \ref{thm:twrcn} gives $O(\Delta (G) \cdot |G|$) bound for
         the \rcn\ of bounded treewidth graphs $G$. As discussed in the
        introduction, the bound is optimal and it  improves on the
        previously known bounds (see Theorems \ref{wood2007} and \ref{tw-rec}).

Since every $k$-simplicial blowup of any graph of treewidth at most $k$ itself has treewidth at most $k$, we get the following immediate corollary of Theorem \ref{thm:twrcn}.

\begin{lem}\label{lemma:twsbrcn}
For every positive integer $k$, every graph of treewidth at most $k$
is $(k,\ k\cdot (k+2))$-agreeable. 
\end{lem}

\subsection{Proof of Theorem~\ref{thm:scgrcnmain}}\label{section:mainproof}

Recall the statement of Theorem~\ref{thm:scgrcnmain}. 
Lemma~\ref{lemma:separatingtrianglesrcn} states that every planar
piece $G_i$ of the decomposition is $(3,3)$-agreeable. Consider the
non-planar pieces of the decomposition. By
Theorem~\ref{thm:robertson2003} of Robertson and Seymour, they have
treewidth at most $t$, where $t\geq 3$, as graphs of treewidth at most
2 are planar \cite{BODLAENDER19981}. Lemma~\ref{lemma:twsbrcn} states
that every treewidth at most $t$ graph is $(t,\ t\cdot
(t+2))$-agreeable. Since every non-planar piece of the decomposition
has treewidth at most $t$ with $t\geq 3$, these pieces are $(3,\
t\cdot (t+2))$-agreeable. Since $t\geq 1$ for all pieces of the
decomposition, if we choose $c:=t\cdot (t+2)$ all the pieces of the
decomposition are $(3,c)$-agreeable. Theorem \ref{thm:clsrcn} (and
Theorem~\ref{thm:robertson2003} by Robertson and Seymour) then implies
that $G$ has \rcn\ at most $3\cdot (t\cdot
(t+2)+2)\cdot\Delta(G)\cdot||G||=3\cdot (t^2+
2t+2)\cdot\Delta(G)\cdot||G||$, as claimed.  

\section{Conclusion and Open Problems}\label{conclusion}

 In this article, we proved that $n$-vertex bounded degree
 single-crossing minor-free graphs have $\Oh{n}$ \rcn.  More strongly
 we proved that for any single-crossing graph $X$,  every  $n$-vertex $X$-minor-free graph $G$ has \rcn\ at most  $\Oh{\Delta(G)\cdot n}$ and the bound is best possible.  The result represents a strong improvement over the previous state of the art on the \rcns\ of minor-closed families of graphs,  as argued in the introduction. 

 The ultimate goal for future work would be to obtain the above result for any fixed graph $X$. For such families $\Oh{f(\Delta)\cdot n}$ bound is not known for any function $f$. In fact, the best known bound on the rectilinear crossing number of bounded degree proper minor-closed families is $\Oh{n \log n}$ \cite{ShahrokhiSSV03}.

 In order to attempt to prove an $\Oh{f(\Delta)\cdot n}$ bound, that is, a linear \rcn\ for all proper minor-closed families of graphs of bounded degree, Robertson and Seymour's graph minors theory tells us that one should provide two ingredients. The first ingredient is to prove the result for $k$-almost embeddable graphs. The second is to be able to handle clique-sums of those. Proving the result for almost embeddable graphs entails proving it for bounded Euler genus graphs, that is, proving a result akin to Theorem \ref{thm:pachtoth2006} by \citet{pach2006crossing} but with the crossing number replaced by the \rcn. However, such a result is not even known for all bounded degree toroidal graphs.

The second ingredient however, handling the clique-sums of \rds, can
be achieved by our Theorem \ref{thm:clsrcn}. In particular, one can
change a definition of $(k,c)$-agreeable to $(k,f(\Delta))$-agreeable
as to allow for any function $f(\Delta)$ and not just the linear
function, $c\cdot \Delta$, and then  recall that the proof of Theorem
\ref{thm:clsrcn} in fact shows that the \rds\ of
$(k,f(\Delta))$-agreeable graphs can be joined by $(\leq k)$-clique
sums into a \rd\ of the resulting graph $G$ while only increasing the
total number of crossings by $2\cdot k \cdot
\Delta(G)\cdot||G||$. Suppose, in the future, one could provide the
first ingredient above, that is, show that almost embeddable
$n$-vertex graphs $G$ have linear \rcn, that is $\RCR{G}\leq
g(\Delta)\cdot n$  for some function $g$.  In that case Lemma
\ref{future}, stated and proved in the appendix, would imply that simplicial blowups of almost embeddable graphs are  $(k,f(\Delta))$--agreeable with $f(\Delta)\in \Oh{\Delta^4}\cdot g(\Delta)$. That and 
 Theorem \ref{thm:clsrcn}, as discussed in this paragraph, would imply that all proper minor-closed families of graphs of bounded degree have linear \rcn.

\newpage

{\fontsize{10pt}{11pt}\selectfont
\bibliographystyle{DavidNatbibStyle}
\bibliography{icalp,DavidBibliography}
}
\section*{Appendix}

\begin{proof}[Proof of Lemma \ref{lemma:hpartitionrcn}]  We start with
  a rectilinear drawing $D(H)$ of $H$ with $\RCR{H}$
  crossings. Consider any vertex (bag) $B$ of $H$. Let
  $C_{\epsilon}(B)$ be a disk of radius $\epsilon >0$ centered at $B$
  in $D(H)$. For each edge $AB$ of $H$, let $C_{\epsilon}(AB)$ be the
  region defined by the union of all the line-segments with one endpoint in $C_{\epsilon}(A)$ and the other in $C_{\epsilon}(B)$. Note that there exists an $\epsilon$ small enough such that all of the following conditions are met:
\begin{itemize}
  \item $C_{\epsilon}(A) \cap C_{\epsilon}(B) = \emptyset$ for all distinct bags $A$ and $B$ of $H$;
  \item $C_{\epsilon}(AB) \cap C_{\epsilon}(PQ) = \emptyset$ for every pair of edges $AB$ and $PQ$ of $H$ that have no endpoints in common and do not cross in $D(H)$;
  \item $C_{\epsilon}(AB) \cap C_{\epsilon}(Q)= \emptyset$ for every triple of distinct bags $A, B, Q$ of $H$ where $AB$ is an edge of $H$;
   \item For each crossing-pair of edges $AB$ and $PQ$ in $D(H)$, $C_{\epsilon}(AB) \cap C_{\epsilon}(PQ)$ is non-empty. We call that region, $C_{\epsilon}(AB) \cap C_{\epsilon}(PQ)$, of the plane \emph{busy region} of pair $AB$ and $PQ$. Finally, the busy regions of all distinct pair of edges are pairwise disjoint. 
\end{itemize}

For each vertex $v$ of $K$ such that $v$ is in a bag $B$ of $H$, draw $v$ as a point in $C_{\epsilon}(B)$ such that the final set of points representing $V(K)$ is in general position. Draw every edge of $K$ straight. This defines a \rd\ $D(K)$ of $K$, since no edge in $D(K)$ contains a vertex other than its own endpoints and no three edges of $D(K)$ cross at one point. 
  
  We first prove that the number of crossings in $D(K)$ is at most $\RCR{H} \cdot w^2\cdot\Delta(K)^2 + (w-1) \cdot \sum_{v\in X} \deg_K(v)^2$ which will prove the first part of the theorem.  Consider two crossing edges $e$ and $f$ in $D(K)$. There are two cases to consider (based on two types of crossings that can occur in $D(K)$).
  \begin{itemize}
  \item Case 1: there is bag $B$ of $H$ that has at least one endpoint of $e$ and at least one endpoint of $f$. Order all the vertices of $B$ = $\{v_1, v_2, \cdots, v_\ell\}$, $l \leq w$ such that $\deg_{K}(v_1) \leq \cdots \leq \deg_{K}(v_\ell)$. Let $v_i$ be an endpoint of $e$ and $v_j$ and endpoint of $f$, $\,i < j$. We charge the crossing between $e$ and $f$ to $v_j$. 
  

  Thus the number of crossings charged to $v_j$ is at most 
				\[ \sum_{i < j} \deg_{K}(v_i) \cdot \deg_{K}(v_j) \leq \sum_{i < j} \deg_{K}(v_j)^2\leq (\ell-1)\cdot \deg_{K}(v_j)^2 \leq (w-1)\cdot \deg_{K}(v_j)^2 \]
				The vertices in the solitary bags of $H$ are charged $0$ crossings, rendering the total number of crossings in Case 1 is at most $(w-1) \sum_{v \in X} \deg_K(v) ^2$. 
				
\item Case 2: there is no bag of $H$ that has both an endpoint of $e$ and an endpoint of $f$. This implies that four endpoints of $e$ and $f$ are in four distinct bags, $A,B,P,Q$ of $H$. Let $e\in C_\epsilon (AB)$ and $f \in C_\epsilon (PQ)$. Since $e$ and $f$ cross, their crossing point must be the busy region of  $AB$ and $PQ$. Denote that region by $R$. There are at most $\Delta(K)\cdot w$ edges of $K$ drawn inside $C_\epsilon (AB)$ that intersect $R$ and at most $\Delta(K)\cdot w$ edges of $K$ drawn inside $C_\epsilon (PQ)$ that intersect $R$. We charge the crossings between these pairs of edges to the busy region $R$. 
%
Thus the number of crossings charged to $R$ is at most $w\Delta(K) \cdot w\Delta(K) = w^2 \cdot \Delta(K)^2$. Since $D(H)$ has $\RCR{H}$ crossings, there are exactly $\RCR{H}$ busy regions determined by crossing edges in $D(H)$. Thus the total number of crossings  in Case 2 is at most $\RCR{H} \cdot w^2 \cdot \Delta (K)^2$.
\end{itemize}
			
Thus $\RCR{K} \leq \RCR{H} \cdot w^2\cdot\Delta(K)^2 + (w-1) \cdot \sum_{v\in X} \deg_K(v)^2$ as stated in part 1.

 We now prove the second part of the theorem. In this case, $H$ is planar. By the Fáry-Wagner theorem \cite{wagner1936bemerkungen, istvan1948straight}, there is a \rd\ $D(H)$ of $H$ with no crossings. Starting with such crossing-free drawing $D(H)$, we produce a \rd\ $D(K)$ of $K$ using the algorithm described above. Let $e$ be an edge of $K$ with an endpoint in some bag $A$ of $H$. We now prove that the number of crossings on $e$ in $D(K)$ is at most $2d$ as claimed in part 2a. There are two cases to consider:
			\begin{itemize}
				\item Case 1: both endpoints of $e$ are in $A$. Then, $e$ is only crossed by the edges that have at least one endpoint in $A$. As there are at most $d$ such edges, there is at most $d$ crossings on $e$ in $D(K)$.
				\item Case 2: the other endpoint of $e$ is in a bag $B$ of $K$ distinct from $A$. Then, since $D(H)$ is crossing-free, $e$ can only be crossed by the edges that have at least one endpoint in $A$ or in $B$. There is at most $2d$ such edges, thus there is at most $2d$ crossings on $e$ in $D(K)$.
			\end{itemize}
			
			In either case, $e$ is crossed by at most $2d$ edges in $D(K)$ as required by part 2a.
			
Finally, consider the case when the non-solitary bags of $H$ form an independent set in $H$. Let $e$ be an edge of $K$. If two endpoints of $e$ are in two distinct solitary bags of $H$ then no edge of $K$ crosses $e$ since $D(H)$ is crossing-free. Therefore, in that case, trivially, there are at most $d$ crossings on $e$ in $D(K)$. Thus we may assume that at least one endpoint of $e$ is in a non-solitary bag of $H$. Let $A$ denote that bag. If the other endpoint of $e$ is also in $A$, the result follows from Case 1 above. Therefore, we may assume that the other endpoint, $v$, of $e$ is in a bag $B$ of $H$ distinct from $A$. $B$ is then a solitary bag (by the independent set assumption). Since the edges incident to the same vertex ($v$ in this case) cannot cross, the only edges that can cross $e$ are those with an endpoint in $A$. There is at most $d$ edges with endpoints in $A$ and thus there are at most $d$ crossings on $e$ in $D(K)$.
\end{proof}

\begin{lem}\label{future}
For every graph $G$ and every $(\leq k)$-simplicial blowup $Q$ of $G$, $\RCR{Q} \leq (\Delta(Q)+1)^2\cdot \RCR{G} + \Delta(Q)^4\cdot ||Q|| $.
\end{lem}
\begin{proof}
Let $S = V(Q)- V(G)$. We now define an $H$-partition of $Q$. To start, we make $H$ isomorphic to $G$ and put each $v \in V(G)$ in the bag $B_v$ in $H$. 
 For each vertex $u\in S$, $u$ is adjacent to all the vertices of some clique $C$ in $G$. Place $u$ in a bag $B_v$ where $v\in C$. This does not change $H$ since $v$ is adjacent to all the neighbours of $u$ in $G$. This defines an $H$-partition of multigraph $Q$. 
  
  For each $v\in H$, each vertex of $S$ in $B_v$ is adjacent to $v$ in $Q$ thus the width of $H$ is at most $\Delta(Q)+1$. Thus by Lemma \ref{lemma:hpartitionrcn}, $\RCR{Q} \leq (\Delta(Q)+1)^4\cdot \RCR{H} + \Delta(Q)^2\cdot ||Q|| $ which is equal to $(\Delta(Q)+1)^4\cdot \RCR{G} + \Delta(Q)^2\cdot ||Q||$ since  $H$ is isomorphic to $G$.
\end{proof}

\end{document}